\documentclass{article} \usepackage{graphicx}
\usepackage{setspace} \usepackage{amsmath}
\usepackage{amssymb,amsfonts,amsbsy} \usepackage{amsthm}
\usepackage{stmaryrd} \usepackage[pdftex]{hyperref}
\usepackage{url} 
\usepackage{enumerate} 

\newcommand{\R}{\mathbb{R}} 
\newcommand{\Z}{\mathbb{Z}} \newcommand{\C}{\mathbb{C}}
\newcommand{\T}{\mathbb{T}} \newcommand{\GT}{\mathcal{W}}
\newcommand{\W}{\mathcal{W}_n}
\newcommand{\bal}{\boldsymbol\theta}
\newcommand{\bth}{\boldsymbol\theta} \newcommand{\al}{\alpha}
\newcommand{\be}{\beta} \newcommand{\la}{\lambda}
\newcommand{\ga}{\gamma} \newcommand{\si}{\sigma}
\newcommand{\prodF}{\prod F(\theta_j^{-1})}
\newcommand{\LR}{c_{\lambda\mu}^{\tau}}
 
\newcommand{\ts}{\tau\sigma^{-1}(j)}
\newcommand{\sgn}{\text{sgn}}
\newcommand{\con}{\text{Con}(\sigma,\tau)}
\newcommand{\conb}{\text{Con}'(\sigma,\tau)}
\newcommand{\ze}{\zeta}

\newtheorem{theorem}{Theorem}[section]
\newtheorem{proposition}[theorem]{Proposition}
\newtheorem{lemma}[theorem]{Lemma}

\begin{document} \title{Two constructions of Markov chains on the dual of $U(n)$} \author{Jeffrey Kuan}

\date{}

\maketitle

\abstract{We provide two new constructions of Markov chains
which had previously arisen from the representation theory of
$U(\infty)$. The first construction uses the combinatorial rule
for the Littlewood--Richardson coefficients, which arise from
tensor products of irreducible representations of the unitary
group. The second arises from a quantum random walk on the von
Neumann algebra of $U(n)$, which is then restricted to the
center. Additionally, the restriction to a maximal torus can be
expressed in terms of weight multiplicities, explaining the
presence of tensor products.}

\section{Introduction} 

This paper will examine a family of Markov processes on the state space 
$
\{\lambda_1 \geq \ldots \geq \lambda_n: \lambda_i \in \Z\}.
$
These Markov processes were first introduced as the Charlier process in \cite{KOR}. In \cite{kn:BF}, the authors introduce a
family of Markov chains on the Gelfand--Tsetlin set
$\mathbb{GT}$. This is the set of infinite sequences
$\la^{(1)}\prec\la^{(2)}\prec\ldots,$ where
$\la^{(k)}=(\la^{(k)}_1\geq\ldots\geq \la^{(k)}_k)$ is a
$k$--tuple of nonincreasing integers and $\lambda\prec\mu$
denotes the condition
$\mu_1\geq\lambda_1\geq\mu_2\geq\lambda_2\geq\ldots\mu_n\geq\lambda_n$.
By considering the map $$ \la^{(1)}\prec\la^{(2)}\prec\ldots
\mapsto \{(\la_i^{(k)}-i,k)\}_{1\leq i\leq k<\infty}\subset
\Z\times\Z_+, $$ these Markov chains define an interacting
particle system on $\Z\times\Z_+$. Drawing lozenges around each
particle yields a random tiling of the half plane. Furthermore,
the condition $\la^{(k)}\prec\la^{(k+1)}$ ensures that there is
a natural interpretation as a randomly growing stepped surface.
This random growth belongs to the $2+1$ anisotropic KPZ class of
stochastic growth models. This universality class is a variant
of the KPZ universality class, which has seen many results in
recent years (see \cite{C} for a survey). By considering
suitable projections, the Markov chains also reduce to TASEP,
PushASEP (introduced in \cite{BF2}) , and the Charlier process
from \cite{KOR}.

The Gelfand--Tsetlin set parametrises the branching of
irreducible representations of the unitary group. Additionally,
the family of Markov chains can be constructed from the
representation theory of the infinite--dimensional unitary group
$U(\infty)$ \cite{BK0}. Tools from representation theory have
yielded a rich variety of two--dimensional dynamics (e.g.
\cite{kn:D2,WW}). One of the most general processes arising from
representation theory are Macdonald processes \cite{BC}.

In this paper, we hope to deepen the connections between
probability theory and representation theory. To this end, we
give two new representation--theoretic constructions for these
Markov chains. The first involves the Littlewood--Richardson
rule for decomposing tensor products of irreducible
representations of $U(n)$. The second involves a quantum random
walk on the von Neumann algebra of $U(n)$, which is then
restricted to the center. These constructions have the advantage
of not requiring infinite--dimensional representation theory and
are therefore more generalisable to other simple Lie groups.

The structure of the paper is as follows. In section \ref{MC},
we review the representation theory of $U(n)$ and introduce the
Markov chains from \cite{kn:BF}. In section \ref{TP}, we provide
a construction the combinatorial description of the
Littlewood--Richardson coefficients. In section \ref{QuantumRW},
we provide another construction, this time using a quantum
random walk on the von Neumann algebra of $U(n)$. This will also
give a representation theoretic explanation (i.e. using tensor
products of representations instead of combinatorics) for the
occurrence of the Littlewood--Richardson coefficients.

\section{Markov chains}\label{MC}

\subsection{Background} Before defining the Markov chains, let
us review some background on the unitary groups. Let $U(n)$
denote the compact group of $n\times n$ unitary matrices.
Occasionally, to clean up notation, $G$ will also refer to
$U(n)$. Let $\T^n\subset U(n)$ be the subgroup of diagonal
unitary matrices, which is a maximal torus of $U(n).$ With
respect to this maximal torus, the weight lattice of $U(n)$ is
easy to describe. The Lie algebra of $\T^n$, denoted $L\T^n$,
consists of imaginary diagonal matrices. The weight lattice
$P\subset (L\T^n)^*$ is the $n$--dimensional lattice generated
by the elements $\epsilon_1,\ldots,\epsilon_n$, where
$\epsilon_j(\text{diag}(u_1,\ldots,u_n))=u_j/(2\pi i)$. Each
$\lambda_1\epsilon_1+\ldots+\la_n\epsilon_n,\la_j\in\Z$ defines
a character of $\T^n$ by sending $(z_1,\ldots,z_n)$ to
$z_1^{\la_1}\cdots z_n^{\la_n}$. In this way, there is an
isomorphism $e:P\rightarrow\widehat{\T^n}$. For $x\in P$ and
$\bth\in\T^n,$ write $x(\theta)=e(x)(\theta).$ Note that with
this notation, $x(\bth)y(\bth)=(x+y)(\bth)$.

The roots of $U(n)$ with respect to $\T^n$ are $e_i-e_j, 1\leq
i\neq j\leq n$. The Weyl group is generated by the reflections
with respect to the roots. It is isomorphic to the group $S_n$
acting on $\{\epsilon_1,\ldots,\epsilon_n\}$, where the
reflection with respect to $e_i-e_j$ is the transposition
$(\epsilon_i \ \epsilon_j)$. The Weyl chamber is thus
$\mathcal{W}_n:=\{\lambda_1\epsilon_1+\ldots+\lambda_n\epsilon_n:\la_1\geq\ldots\geq\la_n,
\la_j\in\Z\}.$

Recall that any irreducible representation of any compact,
connected, simple Lie group is generated by a highest weight
vector, which must lie in the Weyl chamber. Conversely, any
weight in the Weyl chamber generates an irreducible
representation by successively applying negative roots. Therefore
the irreducible unitary representations of $U(n)$ is
parameterised by $\W$.

Let $m_1^{n_1}m_2^{n_2}\ldots$ denote the sequence
$(\underbrace{m_1,\ldots,m_1}_{n_1},\underbrace{m_2,\ldots,m_2}_{n_2},\ldots)$.
For $\lambda,\mu\in\GT_n$, let $\lambda\prec\mu$ denote the
condition
$\mu_1\geq\lambda_1\geq\mu_2\geq\lambda_2\geq\ldots\mu_n\geq\lambda_n$.

For each $\lambda\in\GT_n$, let $\pi_{\la}:U(n)\rightarrow
GL(V_{\la}), \chi_{\lambda}$ and $\dim\lambda$ denote the
corresponding representation, character and dimension. Let
$\widetilde{\chi}_{\la}$ denote the normalized character
$\chi_{\la}/\dim\la$. Recall that the conjugacy class of a
matrix in $U(n)$ is given by its eigenvalues. Therefore,
$\chi^{\lambda}$ is a function of
$\boldsymbol{\theta}=(\theta_1,\ldots,\theta_n)$. Explicitly,
$\chi^{\lambda}$ is just the Schur polynomial $s_{\lambda}$.
Useful formulae are \begin{equation}\label{CharacterFormula}
\chi_{\lambda}(\theta_1,\ldots,\theta_n)=s_{\lambda}(\theta_1,\ldots,\theta_n)=\frac{\det[\theta_i^{\lambda_j+n-j}]_{1\leq
i,j\leq n}}{\det[\theta_i^{n-j}]_{1\leq i,j\leq n}}.
\end{equation} and \begin{equation}\label{Giambelli}
\chi_{\lambda}(\theta_1,\ldots,\theta_n)=s_{\lambda}(\theta_1,\ldots,\theta_n)=\det[h_{\lambda_i-i+j}(\boldsymbol{\theta})],
\end{equation} where $h_k$ is the $k$--th complete homogeneous
symmetric polynomial: $$ h_k(\boldsymbol{\theta}) = \sum_{1\leq
i_1\leq \cdots \leq i_k \leq n} \theta_{i_1}\cdots\theta_{i_k}.
$$ Equation \eqref{Giambelli} is called \textit{the first
Jacobi--Trudi formula}. The elementary homogeneous symmetric
polynomial $e_k$ will also appear: $$ e_k(\boldsymbol{\theta}) =
\sum_{1 < i_1 < \cdots < i_k < n}
\theta_{i_1}\cdots\theta_{i_k}. $$ Note that \eqref{CharacterFormula} implies that
$$
s_{\lambda+\mathbf{1}}(\theta_1,\ldots,\theta_n) = \theta_1\cdots \theta_n s_{\lambda}(\theta_1,\ldots,\theta_n)
$$
where $\mathbf{1}=(1,\ldots,1)$. It is also true that $$
\frac{\det[\theta_i^{\lambda_j+n-j}]_{1\leq i,j\leq n}}{\det[\theta_i^{n-j}]_{1\leq i,j\leq n}} = \begin{cases}
h_k, \ \ \text{when} \ \ \lambda=k0^{n-1}\\ e_k, \ \ \text{when}
\ \ \lambda=1^k0^{n-k}\\
h_k(\theta_1^{-1},\ldots,\theta_n^{-1}), \ \ \text{when}\ \
\lambda=0^{n-1}(-k) \\ \theta_1^{-1}\ldots
\theta_n^{-1}e_{n-k}=e_k(\theta_1^{-1},\ldots,\theta_n^{-1}), \
\ \text{when}\ \ \lambda=0^{n-k}(-1)^k \end{cases} $$ The first two lines are well known, and the second two can be deduced from the first two. A formula for
the dimension is $$ \dim\la = \prod_{i<j}
\frac{\la_i-i-(\la_j-j)}{j-i}, $$ which extends formally to
$\dim:P\rightarrow\R$.

Let $L^2(G,dg)^G$ denote the square--integrable class functions
on $G$. By the Peter--Weyl theorem, $\{\chi_{\la}\}_{\la\in \W}$
is an orthonormal basis for $L^2(G,dg)^G$. For any $\kappa\in
L^2(G,dg)^G$ and any $\la\in\W,$ let $\widehat{\kappa}(\la)$ be
the Fourier coefficient $$ \widehat{\kappa}(\la) = \int_G
\kappa(g)\overline{\chi_{\la}}(g)dg, $$ so that
\begin{equation}\label{FT} \kappa(g) = \sum_{\la\in\W}
\widehat{\kappa}(\la)\chi_{\la}(g). \end{equation} Formally,
this means that $\sum_{\la\in\W}
\overline{\chi_{\la}(g)}\chi_{\la}(g')$ is the Dirac delta
function $\delta_{g^{-1}g'}$.

\subsection{Markov chains} Now review the Markov chains from
\cite{kn:BF}. Let $\theta_1,\ldots,\theta_n$ be fixed nonzero
complex numbers, and let $F$ be an analytic function in an
annulus $A$ which contains all the $\theta_j^{-1}$ such that
each $F(\theta_j^{-1})$ is nonzero. Given such an $F$, define
\begin{equation}\label{DefnOff} f(m):=\frac{1}{2\pi i}\oint
\frac{F(z)}{z^{m+1}}dz, \end{equation} where the integral is
taken over any positively oriented simple loop in $A$. Section
2.3 of \cite{kn:BF} defines matrices $T_n$ with rows and columns
parameterised by $\GT_n$: $$
T_n(\boldsymbol\theta;F)(\lambda,\mu):=\frac{s_{\mu}(\boldsymbol{\theta})}{s_{\lambda}(\boldsymbol{\theta})}
\frac{\det[f(\lambda_j+j-\mu_i-i)]_1^n}{\prodF} $$

\begin{proposition}\label{Multiply} There is the commuting
relation
$T_n(\boldsymbol{\theta};F_1)T_n(\boldsymbol{\theta};F_2)=T_n(\boldsymbol{\theta};F_1F_2)$.
For $\boldsymbol\theta=(1,1,\ldots,1)$,
$T_n(\boldsymbol\theta;F)$ is a stochastic matrix.
\end{proposition} \begin{proof} Proposition 2.10 of \cite{kn:BF}
gives the commuting relation. Proposition 2.8 of \cite{kn:BF}
gives the stochastic matrix result. \end{proof}

Let us now describe the functions $F$ to be considered. Define
the functions \begin{eqnarray*} &F_{\al^+,q}(z)=(1-qz)^{-1}, \ \
&F_{\al^-,q}(z) = (1-qz^{-1})^{-1}, \ \ 1>q\geq 0\\
&F_{\be^+,p}(z)=1+pz, \ \ &F_{\be^-,p}(z)=1+pz^{-1}, \ \ 1 \geq
p\geq 0.\\ &F_{\ga^+,t}(z)=e^{tz},\ \
&F_{\ga^-,t}(z)=e^{tz^{-1}},\ \ t\geq 0. \end{eqnarray*}

\begin{lemma}\label{2.12} For these functions,

\begin{align*}
T_n(\boldsymbol{\theta};F_{\be^-,p})(\lambda,\mu)=\frac{p^k}{\prodF}\frac{s_{\mu}(\bth)}{s_{\la}(\bth)},
\ \ &\text{if each} \ \mu_j-\la_j\in\{0,1\}\ \text{and} \ \sum
(\mu_j-\la_j)=k,\\ 0, \ \ &\text{otherwise}. \end{align*}

\begin{align*}
T_n(\boldsymbol{\theta};F_{\be^+,p})(\lambda,\mu)=\frac{p^k}{\prodF}\frac{s_{\mu}(\bth)}{s_{\la}(\bth)},
\ \ &\text{if each} \ \mu_j-\la_j\in\{-1,0\}\ \text{and} \ \sum
(\mu_j-\la_j)=-k,\\ 0, \ \ &\text{otherwise}. \end{align*}

\begin{align*}
T_n(\bth;F_{\alpha^-,q})(\lambda,\mu)=\frac{q^k}{\prodF}\frac{s_{\mu}(\bth)}{s_{\la}(\bth)},
\ \ &\text{if} \ \la\prec\mu\ \text{and} \ \sum
(\mu_j-\la_j)=k,\\ 0, \ \ &\text{otherwise}. \end{align*}

\begin{align*}
T_n(\bth;F_{\alpha^+,q})(\lambda,\mu)=\frac{q^k}{\prodF}\frac{s_{\mu}(\bth)}{s_{\la}(\bth)},
\ \ &\text{if} \ \mu\prec\la\ \text{and} \ \sum
(\mu_j-\la_j)=-k,\\ 0, \ \ &\text{otherwise}. \end{align*}

\end{lemma} \begin{proof} These are Lemmas 2.12 and 2.13 from
\cite{kn:BF}. \end{proof}

Use the variable $\xi$ to denote one of the symbols
$\al^{\pm},\be^{\pm},\ga^{\pm}$. For $\xi=\be^{\pm}$ and
$\xi=\ga^{\pm}$, the process with transition probabilities
$T_n(\boldsymbol{1},F_{\xi})$ are respectively the
\textit{Krawtchouk} and \textit{Charlier} processes from
\cite{KOR}. These can be described respectively as the Doob
$h$--transform (where $h(\la)=\dim\la$) of $n$ independent
Bernoulli walks and $n$ independent exponential random walks of
rate $1.$

There is a general construction for building multivariate Markov
chains out of $\{T_n:n=1,2,3,\ldots\}$. This construction
requires a intertwining relation between the transition
probabilities (see section 2 of \cite{kn:BF}). It is still an
open problem to find a representation--theoretic interpretation
of the commutation relation.

Let $\mathbb{P}_n(\bth;F)(\mu)=T_n(\bth;F)(\boldsymbol{0},\mu)$.
From Lemma \ref{2.12}, we see that for
$F=F_{\al^{\pm}},F_{\be^{\pm}}$, these are geometric random
variables weighted by the dimension of the representation. The
construction then proceeds as follows. First, for
$F=F_{\al^{\pm}}$ or $F_{\be^{\pm}}$, construct $T_n(F)$ out of
$\mathbb{P}_n(F)$ (Lemmas \ref{FirstStepTensor} and
\ref{FirstStepQuantum} below). Second, we show that there is a
commuting relation that is analogous to the one in Proposition
\ref{Multiply} (Proposition \ref{SecondStepTensor} and Lemma
\ref{SecondStepQuantum} below). Finally, use a continuity
argument (Lemmas \ref{ThirdStepTensor} and
\ref{ThirdStepQuantum} below) to give $F_{\ga^{\pm}}$.

\section{Combinatorial Formula}\label{TP} 
From Lemma 3.4 and (2) of \cite{BK0}, the measures $\mathbb{P}_n(\boldsymbol{\theta};F)(\mu)$ have a nice algebraic interpretation from the formula
$$
\sum_{\mu \in \mathcal{W}_n} \mathbb{P}_n(\boldsymbol{\theta};F)(\mu) \frac{s_{\mu}(\boldsymbol{\theta}) }{ s_{\mu}(\boldsymbol{1})} = F(\theta_1) \cdots F(\theta_n).$$
In words, the measures $\mathbb{P}_n(\boldsymbol{\theta};F)$ are the coefficients in the decomposition of the function $F(\theta_1)\cdots F(\theta_n)$ over the normalized irreducible characters of $U(n)$. However, these measures are obtained by applying $T_n(\boldsymbol{\theta};F)$ to the initial condition $\boldsymbol{0}$, and it is not clear what algebraic structure remains if started from a different initial condition. Theorem \ref{Theorem} below will provide a generalization to the case when the initial condition is some nonzero $\lambda \in \mathcal{W}_n$. It also worth noting that for $F=F_{\alpha^+}$, a similar theorem has appeared in the case of the orthogonal groups, using Pieri's formula (see \cite{kn:D},\cite{kn:K}). The theorem has also found use in \cite{K2}.

Before proceeding to the statement of Theorem \ref{Theorem}, let us briefly recall the Littlewood--Richardson coefficients. For any
two paritions $\lambda,\tau$ such that $\lambda_j\leq\tau_j$ for
each $j$, the \textit{skew diagram} of $\tau\backslash\lambda$
is the set--theoretic difference of the Young diagrams of
$\lambda$ and $\tau$. A \textit{skew Tableau} of shape
$\tau\backslash\lambda$ and weight $\mu$ is obtained by filling
in the skew diagram of $\tau\backslash\lambda$ with positive
integers such that the integer $k$ appears $\mu_k$ times. A skew
Tableau is \textit{semistandard} if it the entries weakly
increase along each row and strictly increase down each column.
A \textit{Littlewood--Richardson tableau} is a semistandard skew
Tableau with the additional property that in the sequence
obtained by concatenating the reversed rows, every initial part
of the sequence contains any number $k$ at least as often as it
contains $k+1$. See figure \ref{Figure} for two examples. The Littlewood--Richardson coefficient $c_{\lambda\mu}^{\tau}$ counts the number of Littlewood--Richardson tableaux of shape $\tau\backslash\lambda$ and of weight $\mu$. In the example in figure \ref{Figure}, $c^{\tau}_{\lambda\mu}$ equals $2$.

In the special case $\mu=k0^{n-1}$, the Littlewood--Richardson
rule is known as \textit{Pieri's formula}. In this case, the
semistandard skew Tableau can only be filled with $1$'s, so the
condition on the concatenated sequence is automatically
satisfied. The only requirement is that the skew diagram of
$\tau\backslash\lambda$ does not contain two boxes in the same
column. In other words, \begin{equation} \text{ when }
\mu=k0^{n-1}, \ \ \LR = \begin{cases} 1, \ \ &\text{if}\
\la\prec\tau \text{ and } \sum (\tau_j-\la_j)=k,\\ 0, \ \
&\text{else}. \end{cases} \label{Pieri1} \end{equation}

We also need the special case $\mu=1^k0^{n-k}$. The integers
appearing in the semistandard skew Tableau are
$\{1,2,\ldots,k\}$, so the condition on the concatenated
sequence can only hold if the skew diagram of
$\tau\backslash\lambda$ does not contain two boxes in the same
\textit{row}. In other words, \begin{equation} \text{ when }
\mu=1^k0^{n-k}, \ \ \LR = \begin{cases} 1, \ \ & \text{if} \
\tau_j-\la_j\in\{0,1\}\ \text{and} \ \sum (\tau_j-\la_j)=k,\\ 0,
\ \ & \text{else}. \end{cases} \label{Pieri2} \end{equation}

\begin{figure} \caption{For $\tau=(4,3,2), \lambda=(2,1,0)$ and
$\mu=(3,2,1)$, here is a Littlewood--Richardson tableau of shape
$\tau\backslash\lambda$ and weight $\mu$. The sequence obtained
by concatenating the reversed rows is $112132$ in the first case, and $112231$ in the second.} \begin{center}
\includegraphics[height=0.7in]{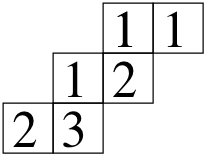} \quad \includegraphics[height=0.7in]{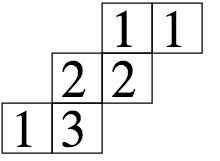}
\end{center}
\label{Figure} \end{figure}

The Littlewood-Richardson coefficients are related to
representation theory by the decomposition $$ V_{\lambda}\otimes
V_{\mu} = \bigoplus_{\tau\in GT_n} \LR V_{\tau}. $$ Since the
character of $V_{\lambda}$ is the Schur polynomial $s_{\la}$ it
is equivalent to say $$ s_{\la}s_{\mu}=\sum_{\tau\in\GT_n} \LR
s_{\tau}. $$ 

Note that this implies the identity $c^{\tau}_{\lambda\mu}=c^{\tau}_{\mu\lambda}$, which would not have been obvious from the combinatorial description. Also note that this allows for a natural definition of $c^{\tau}_{\mu\lambda}$ when one of $\tau,\mu,\lambda$ has negative entries. Namely, letting $\mathbf{1}=(1,1,\ldots,1)$, since $s_{\lambda+\mathbf{1}}(x_1,\ldots,x_n)= x_1\cdots x_n s_{\lambda}(x_1,\ldots,x_n)$, we can define 
$$
c^{\tau-\mathbf{1}}_{\lambda-\mathbf{1},\mu}=c^{\tau}_{\lambda\mu}.
$$
Also define the coefficients
$c_{\lambda\sigma\mu}^{\tau}$ by $$
s_{\la}s_{\si}s_{\nu}=\sum_{\tau\in\GT_n}c_{\lambda\sigma\nu}^{\tau}s_{\tau}.
$$ It follows immediately that $$ \sum_{\mu\in\GT_n} \LR
c_{\si\nu}^{\mu}=c_{\la\si\nu}^{\tau}. $$

The theorem can now be stated.

\begin{theorem}\label{Theorem} For $F=F_{\xi}$,
\begin{equation}\label{Star} \sum_{\mu\in\GT_n}
\mathbb{P}_n(\bth,F)(\mu)\LR\frac{s_{\tau}(\bth)}{s_{\la}(\bth)s_{\mu}(\bth)}
= T_n(\bth,F)(\lambda,\tau). \end{equation} \end{theorem}

\begin{proof} Let $\mathcal{C}$ be the set of all functions
$F:A\rightarrow\C$ such that \eqref{Star} holds.
\begin{lemma}\label{FirstStepTensor} The functions
$F=F_{\al^{\pm}},F_{\be^{\pm}}$ are in $\mathcal{C}$.
\end{lemma} \begin{proof} Start with $1+pz^{-1}$. By lemma
\ref{2.12}, \begin{equation}\label{P1}
\mathbb{P}_n(\bal;F)(\mu)= \begin{cases}
s_{\mu}(\bal)\frac{p^k}{\prodF}, \ \ & \mu=1^k0^{n-k},\\ 0, \ \
&\text{else}. \end{cases} \end{equation} Thus it suffices to
consider the value of $\LR$ when $\mu=1^k0^{n-k}$. By
using Pieri's formula \eqref{Pieri2} and another application of
lemma \ref{2.12}, $F_{\be^+}\in\mathcal{C}$.

Now let consider $1+pz$. Since $\tilde{f}(m)=f(-m)$, \[
\mathbb{P}_n(\bal;F_{\be^-})(\mu)= \begin{cases}
s_{\mu}(\bal)\frac{p^k}{\prodF}, \ \ & \mu=0^{n-k}(-1)^k,\\ 0, \
\ &\text{else}. \end{cases} \] Since
$\LR=c^{\tau+\mathbf{1}}_{\mu+\mathbf{1},\la}$, then for
$\mu=0^{n-k}(-1)^k$, \[ \LR = \begin{cases} 1, \ \ & \text{if
each} \ \tau_i-\la_i\in\{-1,0\}\ \text{and} \ \sum
(\tau_j-\la_j)=-k,\\ 0, \ \ & \text{else}. \end{cases} \]
Therefore, by lemma \ref{2.12}, $F_{\be^-}\in\mathcal{C}$.

Now consider the function $F(z)=(1-qz^{-1})^{-1}$ By lemma
\ref{2.12}, \[ \mathbb{P}_n(\bal;F)(\mu)= \begin{cases}
s_{\mu}(\bal)\frac{q^k}{\prod F(\alpha_j^{-1})}, \ \ &
\mu=k0^{n-1},\\ 0, \ \ &\text{else}. \end{cases} \] By
\eqref{Pieri1} and lemma \ref{2.12}, $F_{\al^+}\in\mathcal{C}$.

Finally let $\tilde{F}(z)=(1-qz)^{-1}$. Then \[
\mathbb{P}_n(\bal;F)(\mu)= \begin{cases}
s_{\mu}(\bal)\frac{q^k}{\prodF}, \ \ & \mu=0^{n-1}(-k),\\ 0, \ \
&\text{else}. \end{cases} \] Using the identity $$
s_{\lambda}(\bal^{-1})=(\theta_1\ldots\theta_n)^{-\lambda_1}s_{(\lambda_1-\lambda_n,\ldots,\la_1-\lambda_2,0)}(\bal),
$$ we get for $\mu=0^{n-1}(-k)$ \begin{eqnarray*}
s_{(\lambda_1-\lambda_n,\ldots,\lambda_{1}-\lambda_2,0)}(\bal)s_{k0^{n-1}}(\bal)&=&
(\theta_1\ldots\theta_n)^{\lambda_1}s_{\la}(\bth^{-1})\cdot
s_{\mu}(\bth^{-1})\\ &=& (\theta_1\ldots\theta_n)^{\lambda_1}
\sum_{\tau} \LR s_{\tau}(\bth^{-1})\\ &=&\sum_{\tau}\LR
(\theta_1\ldots\theta_n)^{\la_1-\tau_1}
s_{(\tau_1-\tau_n,\ldots,\tau_{1}-\tau_2,0)}(\bth) \\
&=&\sum_{\tau}\LR s_{(\la_1-\tau_n,\ldots,\la_1-\tau_1)}(\bth),
\end{eqnarray*} so \[ \LR = \begin{cases} 1, \ \ & \text{if} \
(\lambda_1-\lambda_n,\ldots,\la_1-\lambda_2,0) \prec
(\la_1-\tau_n,\ldots,\la_{1}-\tau_1) \ \text{and} \ \sum
(-\tau_j+\la_j)=k,\\ 0, \ \ & \text{else}. \end{cases} \]
Equivalently, $$ \LR= \begin{cases} 1, \ \ & \text{if}\
\tau\prec\lambda \text{ and } \sum (\tau_j-\la_j)=-k\\ 0,\ \ &
\text{else}. \end{cases} $$ Therefore, by lemma \ref{2.12},
$F_{\al^-}\in\mathcal{C}$.

\end{proof}

\begin{proposition}\label{SecondStepTensor} If
$F_1,F_2\in\mathcal{C}$, then $F_1F_2\in\mathcal{C}$.
\end{proposition} 
\begin{proof} 
Fix $\lambda$ and $\tau$. We want to prove that
\begin{equation*} \sum_{\mu\in\GT_n}
\mathbb{P}_n(\bth,F_1F_2)(\mu)\LR\frac{s_{\tau}(\bth)}{s_{\la}(\bth)s_{\mu}(\bth)}
= T_n(\bth,F_1F_2)(\lambda,\tau). \end{equation*} 
Consider the expression
$$
P(\sigma,\gamma,\nu) =
\mathbb{P}_n(\bal;F_1)(\si)c_{\la\si}^{\ga}\frac{s_{\ga}(\bal)}{s_{\la}(\bal)s_{\si}(\bal)}\mathbb{P}_n(\bal;F_2)(\nu)c^{\tau}_{\ga\nu}\frac{s_{\tau}(\bal)}{s_{\ga}(\bal)s_{\nu}(\bal)}.
$$ 
Since $F_1,F_2\in \mathcal{C}$, we have that
\begin{align*}
\sum_{\sigma,\gamma,\nu \in \mathcal{W}_n} P(\sigma,\gamma,\nu) &= \sum_{\gamma \in \mathcal{W}_n} T_n(\bal,F_1)(\lambda,\gamma)T_n(\bal,F_2)(\gamma,\tau)\\
&= T_n(\bal,F_1F_2)(\lambda,\tau)
\end{align*}
Now we also claim that 
\begin{equation}\label{done}
\mathbb{P}_n(\bal;F_1F_2)(\mu)\LR\frac{s_{\tau}(\bal)}{s_{\la}(\bal)s_{\mu}(\bal)} = \sum_{\sigma,\gamma,\nu \in \mathcal{W}_n} \dfrac{c_{\si\nu}^{\mu}\LR}{c_{\la\si\nu}^{\tau}}P(\sigma,\gamma,\nu) .
\end{equation}
Summing over $\ga$ on the right--hand--side yields
$$
\sum_{\sigma,\nu \in \mathcal{W}_n} \dfrac{c_{\si\nu}^{\mu}\LR}{c_{\la\si\nu}^{\tau}}\mathbb{P}_n(\bal;F_1)(\si)c_{\la\si\nu}^{\tau}\frac{1}{s_{\la}(\bal)s_{\si}(\bal)}\mathbb{P}_n(\bal;F_2)(\nu)\frac{s_{\tau}(\bal)}{s_{\nu}(\bal)},
$$
while the left--hand--side equals (since $F_2\in \mathcal{C}$)
\begin{align*}
&\mathbb{P}_n(\bal;F_1)(\si)T_n(\bal;F_2)(\si,\mu) \LR\frac{s_{\tau}(\bal)}{s_{\la}(\bal)s_{\mu}(\bal)} \\
&=\sum_{\si\in\GT_n} \mathbb{P}_n(\bal;F_1)(\si)
\sum_{\nu\in\GT_n}
\mathbb{P}_n(\bal;F_2)(\nu)c_{\si\nu}^{\mu}\LR\frac{s_{\mu}(\bal)}{s_{\si}(\bal)s_{\nu}(\bal)}\frac{s_{\tau}(\bal)}{s_{\la}(\bal)s_{\mu}(\bal)} ,
\end{align*} 
which proves \eqref{done}. And then summing both sides of \eqref{done} over $\mu\in \mathcal{W}_n$ yields the result.
\end{proof}

\begin{lemma}\label{ThirdStepTensor} If $\{F_k\}$ is a sequence
of functions in $\mathcal{C}$ which converges to $F$ uniformly
in $A$, then $F\in\mathcal{C}$. \end{lemma} \begin{proof} It is
immediate that $\{f_k\}$ converges to $f$ uniformly. Since the
determinant is a continuous function of its entries,
$T_n(\bal;F_k)(\la,\tau)$ converges to $T_n(\bal;F)(\la,\tau)$.
Since sum in the left--hand side of \eqref{Star} only has
finitely many terms, convergence must hold as well. \end{proof}

Finally, since $e^{x}=\lim_{k\rightarrow\infty} (1+ x/k)^k =
\lim_{k\rightarrow\infty} (1-x/k)^{-k}$, Lemma
\ref{FirstStepTensor}, Proposition \ref{SecondStepTensor} and
Lemma \ref{ThirdStepTensor} prove Theorem \ref{Theorem}.
\end{proof}

\section{Quantum Random Walk}\label{QuantumRW}

Let us introduce some notation, which will follow \cite{kn:B}
closely.

Let $G$ be a compact topological group, let $dg$ denote its Haar
measure (normalized to have total weight $1$), and let
$H=L^2(G,dg)$ be the Hilbert space of square--integrable
functions. Let $\alpha$ denote the representation of $G$ on $H$
by left translation. In other words, for $f\in B(H)$ a unitary
operator on $H$, the map $\alpha:G\rightarrow B(H)$ is defined
by $[\al(g)(f)](x)=f(xg)$. The von Neumann algebra of $G$,
denoted $vN(G)$, is the closure (under the strong operator
topology) of the $*$--subalgebra of $B(H)$ generated by
$\alpha(G)$.

Let $\kappa$ be a continuous, positive type function on $G$
which sends the identity to $1$. This defines a state $\varphi$
on $vN(G)$ by $\varphi(\alpha(g))=\kappa(g)$, and also defines a
completely positive map $Q_n(\kappa)$ on $vN(G)$ by
$[Q_n(\kappa)](\alpha(g))=\kappa(g)\alpha(g)$.

Since $vN(G)$ is a unital $C^*$--algebra, we can define the
infinite tensor product $vN(G)^{\otimes \infty}$, which is also
a $C^*$--algebra. Let $\varphi^{\otimes\infty}$ be the state on
$vN(G)^{\otimes\infty}$ defined by
$\varphi^{\otimes\infty}(x_1\otimes
x_2\otimes\cdots)=\varphi(x_1)\varphi(x_2)\ldots$. The
Gelfand--Naimark--Segal construction produces a von Neumann
algebra $\mathcal{A}$. For nonnegative integers $k$, define
$j_k: vN(G)\rightarrow \mathcal{A}$ by $j_0(\alpha(g))=Id$.
and $j_k(\alpha(g))=\alpha(g)^{\otimes k}\otimes Id\otimes Id
\otimes \cdots$. The $j_k$ form what is called a
``non-commutative Markov process''. There is a projection map
$E_k$ from $\mathcal{A}$ to $\mathcal{A}_k$, the von Neumann
subalgebra generated by the images of $j_0,\ldots,j_k$. For
$k\leq m$, there is the Markov property $E_k\circ j_m=j_k\circ
Q_n^{m-k}$. One could think of these objects with the following
analogy: \begin{center} \begin{tabular}{ | l | l | l | l | l | l
| l |} \hline Classical & State Space $S$ &
$(\Omega,\mathcal{F})$ & $(\Omega,\mathcal{F}_k)$ &
$X_k:\Omega\rightarrow S$ &
$\mathbb{E}(\cdot\vert\mathcal{F}_k)$ & $\mathbb{E} $ \\\hline
Quantum & $vN(G)$ & $\mathcal{A}$ & $\mathcal{A}_k$ & $j_k$ &
$E_k(\cdot)$ & $\varphi^{\otimes\infty} $ \\ \hline
\end{tabular} \end{center}

\subsection{Restriction to Center} Let $Z(vN(G))$ be the center
of $vN(G)$. The Peter--Weyl theorem gives an isomorphism
$\chi:Z(vN(G))\rightarrow L^{\infty}(\widehat{G})$, where
$\widehat{G}$ is the set of equivalence classes of irreducible
representations of $G$. If $\kappa$ is constant on conjugacy
classes, then $Q_n$ sends $Z(vN(G))$ to itself. Because it is
completely positive, the map $\chi\circ Q_n \circ \chi^{-1}$
defines a transition matrix for a (classical) Markov chain with
state space $\widehat{G}$, assuming that $\kappa$ is normalized so that $\kappa(\mathrm{Id})=1$. By a slight abuse of notation, let
$Q_n(\kappa)(x,y)$ denote the transition probabilities.

Now let $G=U(n)$. Define $\kappa_{F}:U(n)\rightarrow\C$ to be
the class function defined by $$ \kappa_{F}(\bth)=\prod_{j=1}^n
\frac{F(\theta_j)}{F(1)}. $$ Here,
$\bth=(\theta_1,\ldots,\theta_n)$ are the eigenvalues of the
unitary matrix on which $\kappa_F$ is applied. If $F=F_{\xi}$,
write $\kappa_{\xi}=\kappa_{F_{\xi}}$.

Here is some useful information about $Q_n(\kappa)$. Below, $\text{Mat}(\W \times \W)$ is the $^*$--algebra of matrices with rows and columns indexed by the Weyl chamber $\W$, and any $M \in \text{Mat}(\W \times \W)$ has $(\lambda,\mu)$--entry denoted by $M(\lambda,\mu)$. Additionally, $BC(G,\C)$ denotes the complex--valued continuous class functions on $G$.

\begin{proposition}\label{QProp} 1. For any $\kappa\in
L^2(G,dg)^G,$ \begin{equation}\label{Qn1}
Q_n(\kappa)(\lambda,\mu) =
\frac{\dim{\mu}}{\dim{\lambda}}\int_{U(n)}
\chi_{\lambda}(g)\overline{\chi_{\mu}(g)}\kappa(g)dg.
\end{equation}

2. The map $Q_n:BC(G,\C)^G\rightarrow\text{Mat}(\W\times\W)$ is a morphism
of $^*$--algebras.

\end{proposition} \begin{proof} 1. This is Theorem 3.2 from
\cite{kn:Bi2}. Although the result there is only stated for
certain $\kappa$, by following the proof one sees that it holds
more generally.

2. The fact that $Q_n$ preserves linearity and $^*$ follows
immediately from \eqref{Qn1}. By applying \eqref{FT} to
\eqref{Qn1}, it is immediate that multiplication is also
preserved. Another way to see this is to use the quantum random
walk: let $Q_1,Q_2$, and $Q_{12}$ be the maps $vN(G)\rightarrow
vN(G)$ defined by sending $\alpha(g)$ to
$\kappa_{1}(g)\alpha(g),$ $\kappa_{2}(g)\alpha(g)$ and
$\kappa_2(g)\kappa_{1}(g)\al(g)$ respectively. By construction,
$Q_n(\kappa_{1}),$ $Q_n(\kappa_{2}),$ and
$Q_n(\kappa_1\kappa_2)$ are the respective restrictions to
$Z(vN(G))$. Since $Q_1\circ Q_2=Q_{12}$, the result follows.

\end{proof}

Now specialize to $\kappa_{\xi}.$ Below, let $\xi^*$ denote $\xi$ with $\pm$ replaced by $\mp$; for example, $(\alpha^+)^*=\alpha^-$. This discrepancy is due to the (non--canonical) choice of the representation $\alpha$ as acting by left translation, rather than by right translation. 

 \begin{theorem}\label{QRW} For
any symbol $\xi$,
$Q_n(\kappa_{\xi})=T_n(\boldsymbol{1},F_{\xi^*})$. 
\end{theorem}
\begin{proof} Start with:

\begin{lemma}\label{FirstStepQuantum} Theorem \ref{QRW} holds
for $\xi=\al^{\pm},\be^{\pm}$. \end{lemma} \begin{proof} By
Weyl's integration formula and \eqref{CharacterFormula},
equation \eqref{Qn1} implies $$ Q_n(\kappa)(\lambda,\mu) =
\frac{\dim\mu}{\dim\lambda}\frac{1}{n!} \int_{\mathbb{T}^n}
\det[\theta_i^{\lambda_j+n-j}]\overline{\det[\theta_i^{\mu_j+n-j}]}\kappa_{}(\theta_1,\ldots,\theta_n)
d\theta_1\cdots d\theta_n. $$ This integral can be written in complex analytic
notation. Since the maximal torus is
$$
\mathbb{T}^n = \{ \mathrm{diag}(z_1,\ldots, z_n) : | z_j | = 1 \} \subset U(n),
$$ 
$\mathbb{T}^n$ can therefore be identified with the $n$--fold product of the unit circle in $\mathbb{C}$, which will also be denoted $\mathbb{T}^n$. The Haar measure $d\theta$ on $\T$ is
$dz/2\pi iz,$ so the integral equals 

\begin{equation}\label{Qn}
Q_n(\kappa)(\lambda,\mu)=\frac{\dim\mu}{\dim\lambda}\frac{1}{n!}\left(\frac{1}{2\pi
i}\right)^n\int_{\mathbb{T}^n}
\det[z_i^{\lambda_j+n-j}]\overline{\det[z_i^{\mu_j+n-j}]}\kappa_{}(z_1,\ldots,z_n)
\frac{dz_1\cdots dz_n}{z_1\cdots z_n}. \end{equation}

Note that \begin{equation}\label{KappaAlpha+}
\kappa_{\al^+}(\boldsymbol{z}) = \prod_{j=1}^n
\frac{(1-qz_j)^{-1}}{(1-q)^{-1}} = \prod_{j=1}^n \frac{1+qz_j +
(qz_j)^2+\ldots}{(1-q)^{-1}} = \sum_{k=0}^{\infty} \frac{q^k
h_k(\boldsymbol{z})}{(1-q)^{-n}}. \end{equation} $$
\kappa_{\al^-}(\boldsymbol{z}) = \prod_{j=1}^n
\frac{(1-qz_j^{-1})^{-1}}{(1-q)^{-1}} = \prod_{j=1}^n
\frac{1+qz_j^{-1} + (qz_j^{-1})^2+\ldots}{(1-q)^{-1}} =
\sum_{k=0}^{\infty} \frac{q^k
h_k(\boldsymbol{z^{-1}})}{(1-q)^{-n}}. $$
and also
\begin{equation}\label{KappaBeta+}
\kappa_{\be^+}(\boldsymbol{z}) = \prod_{j=1}^n
\frac{1+pz_j}{1+p} = \sum_{k=0}^{\infty} \frac{p^k
e_k(\boldsymbol{z})}{(1+p)^n}. \end{equation} $$
\kappa_{\be^-}(\boldsymbol{z}) = \prod_{j=1}^n
\frac{1+pz_j^{-1}}{1+p} = \sum_{k=0}^{\infty} \frac{p^k
e_k(\boldsymbol{z}^{-1})}{(1+p)^n}. $$

By expanding the determinant in \eqref{Qn},
\begin{multline}\label{Q} Q_n(\kappa_{\xi})(\lambda,\mu)=
\frac{\dim\mu}{\dim\lambda}\frac{1}{n!}\left(\frac{1}{2\pi
i}\right)^n\int_{\mathbb{T}^n} \left (\sum_{\sigma\in
S_n}\text{sgn}(\si)z_{\si(1)}^{\lambda_1+n-1}\cdots
z_{\si(n)}^{\lambda_n}\right) \\ \times\left(\sum_{\tau\in
S_n}\text{sgn}(\tau)z_{\tau(1)}^{-\mu_1-n+1}\cdots
z_{\tau(n)}^{-\mu_n}\right)\kappa_{\xi}(z_1,\ldots,z_n)\frac{dz_1\cdots
dz_n}{z_1\cdots z_n} \end{multline} Observe that for any $\xi$,
the only contributions to the integral come from the constant
terms after expanding the product.

First consider when $\xi=\alpha^+$. Define the
\textit{contribution} from $\si$ and $\tau$ to be $$ \con=
\begin{cases} \sgn(\si)\sgn(\tau), \ \ &\text{if each} \ \
\lambda_j-j\leq \mu_{\tau^{-1}(\si(j))}-\tau^{-1}(\si(j)),\\ 0,
\ &\text{else}. \end{cases} $$ 
In words, given the two decreasing sequences $\lambda_j-j$ and $\mu_j-j$, we are looking at permutations of the second one which are entrywise smaller than the first one. Also observe that $\con$ depends only on $\tau^{-1} \sigma$ and thus  
\[\sum_{\sigma,\tau} \con= n! \sum_{\pi} \mathrm{Con}(\pi)\]
where, by definition, $\mathrm{Con}(\pi)=\con$ for any pair $(\sigma,\tau)$ with $\tau^{-1} \sigma=\pi$. 
By \eqref{KappaAlpha+} $$
Q_n(\kappa_{\al^+})(\lambda,\mu)=\frac{\dim\mu}{\dim\lambda}\frac{1}{n!}
\frac{q^{\sum_{i=1}^n \mu_i-\lambda_i}}{(1-q)^{-n}}
\sum_{\sigma,\tau}\con . $$

Proceed with three steps: \begin{enumerate}[I] \item If
$\mu_i<\lambda_i$ for some $1\leq i\leq n$, then $\con=0$ for
all $\sigma,\tau\in S_n$. \item If $\mu_i>\lambda_{i-1}$ for
some $1<i\leq n$, then $Q_n(\kappa)(\lambda,\mu)=0$. \item If
$\lambda\prec\mu$, then $\sum_{\pi} \mathrm{Con}(\pi)= 1$.
\end{enumerate}

For I, fix an $i$ such that $\mu_i<\lambda_i$ and suppose $\con$
is nonzero for some $\sigma,\tau\in S_n$. If there is a $j<i$
which satisfies $\tau(\sigma^{-1}(j))\geq i$, then
$\lambda_j-j\leq\mu_{\tau\sigma^{-1}(j)}-\tau\sigma^{-1}(j)\leq\mu_i-i$,
implying that $\mu_i > \lambda_j \geq \lambda_i$, which
contradicts $\mu_i<\lambda_i$. Therefore $\tau\sigma^{-1}$ sends
the set $\{1,2,\ldots,i-1\}$ to itself, so
$\tau(\si^{-1}(i))\geq i$. Thus $\lambda_i-i \leq
\mu_{\tau\sigma^{-1}(i)} - \tau\sigma^{-1}(i) \leq \mu_i-i$, so
$\lambda_i\leq\mu_i$. Again, this is a contradiction. Therefore,
the only possibility is that all $\con$ are zero.

For II, suppose that $Q_n(\kappa)(\lambda,\mu)\neq 0$. Fix an $i$
such that $\mu_i>\lambda_{i-1}$. I claim that for some $j\leq
i$, there is no $k$ such that
$\mu_j-j<\lambda_k-k\leq\mu_{j-1}-(j-1)$. This is simply because
there are $i-1$ intervals $(\mu_j-j,\mu_{j-1}-(j-1)]$, but only
$i-2$ numbers $\lambda_k-k$ that can fit into these intervals,
so at least one interval must be empty. The claim implies that
the inequality $\lambda_k-k\ \leq \mu_j-j$ holds if and only if
the inequality $\lambda_k-k\leq \mu_{j-1}-(j-1)$ holds.
Therefore $\con+\text{Con}(\sigma,(j\ \ \ j-1)\cdot\tau)=0$, so
the sum $\sum_{\si,\tau}\con$ is zero.

For III, suppose that $\con\neq 0$. Then, using that
$\lambda\prec\mu$, a strong induction argument on $j$ implies
that $\tau^{-1}\sigma(j)=j$ for all $j$. In other words,
$\con\neq 0$ implies that $\si=\tau$. Since the converse is
immediate, the sum $\sum_{\si,\tau} \con$ simplifies to
$\sum_{\si\in S_n} \text{Con}(\si,\si)$, which equals $\vert
S_n\vert = n!$.

Together, I, II and III imply that $$
Q_n(\kappa_{\al^+})(\lambda,\mu)=\frac{q^{\sum_{i=1}^n
\mu_i-\la_i}}{(1-q)^{-n}}
\frac{\dim\mu}{\dim\lambda}1_{\lambda\prec\mu}, $$ which is just
$T_n(\boldsymbol{1},F_{\al^-})$.

Now move on to $\xi=\be^+$. Define the \textit{contribution}
from $\si$ and $\tau$ to be $$ \conb = \begin{cases}
\sgn(\si)\sgn(\tau), \ \ &\text{if each}\ \
\la_j-j-(\mu_{\ts}-\ts)\in \{0,-1\}\\ 0, \ &\text{else},
\end{cases} $$ so that 
\begin{align*}
Q_n(\kappa_{\be^+})(\lambda,\mu)&=\frac{\dim\mu}{\dim\lambda}\frac{1}{n!}
\frac{p^{\sum_{i=1}^n \mu_i-\lambda_i}}{(1+p)^n}
\sum_{\sigma,\tau}\conb \\
&=\frac{\dim\mu}{\dim\lambda}
\frac{p^{\sum_{i=1}^n \mu_i-\lambda_i}}{(1+p)^n}
\sum_{\pi} \mathrm{Con}'(\pi), 
\end{align*}
where as before $\mathrm{Con}'(\tau\sigma^{-1})=\mathrm{Con}'(\sigma,\tau)$.

Again we prove three steps: \begin{enumerate} \item If
$\mu_i<\lambda_i$ for some $1\leq i\leq n$, then $\conb=0$ for
all $\sigma,\tau\in S_n$. \item If $\mu_i-\lambda_{i}\notin
\{0,1\}$ for some $1<i\leq n$, then
$Q_n(\kappa)(\lambda,\mu)=0$. 
\item If all $\mu_i-\lambda_i$ are
$0$ or $1$, then $\sum_{\pi} \mathrm{Con}'(\pi)= 1$.
\end{enumerate}

For 1, notice that $\con=0$ implies that $\conb=0$, and I above
shows that $\con$ is always $0$.

For 2, we already know that $Q_n(\kappa)(\la,\mu)=0$ if some
$\mu_i<\la_i$, so we can assume that all $\mu_i\geq\la_i$. Now
fix some $j$ such that $\mu_j-\la_j\geq 2$, and suppose
$\conb\neq 0$. Then $\ts\leq j$ would imply $\mu_{\ts}-\ts \geq
\mu_j-j$, which implies that
$\la_j-j-(\mu_{\ts}-\ts)\leq\la_j-\mu_j\leq -2$, which
contradicts $\conb\neq 0$. So $\ts>j$. Thus there must be some
$i>j$ such that $\tau\si^{-1}(i)\leq j$ (or else $\tau\si^{-1}$
would map $\{j,\ldots,n\}$ to $\{j+1,\ldots,n\}$). This implies
that $\la_i-i < \la_j - j \leq \mu_j-j-2 \leq
\mu_{\tau\si^{-1}(i)}-\tau\si^{-1}(i) -2$, which again
contradicts $\conb\neq 0$. Thus, $\conb$ must always be zero.

For 3, suppose that $\mathrm{Con}'(\pi)\neq 0$ with $\pi \neq \mathrm{id}$, and let
$j$ be the smallest integer such that $\pi(j)\neq j$. Then
$\pi(j)>j$ and there is some $i>j$ such that $\pi(i)=j$.
This implies that $\la_i-i < \la_j - j \leq \mu_{\pi(j)}-\pi(j)<
\mu_j-j =\mu_{\pi(i)}-\pi(i)$, which implies
that
$\lambda_i-i-(\mu_{\pi(i)}-\pi(i))\leq
-2$, which is a contradiction. Therefore $\mathrm{Con}'(\pi)=1$ exactly when
$\pi$ is the identity permutation, and the result follows.

For the $\alpha^-$ case, it is almost identical to the
$\alpha^+$ case.

Now move on to the $\beta^-$ case. Since $$
e_k(\boldsymbol{z}^{-1})=z_1^{-1}\ldots z_n^{-1}
e_{n-k}(\boldsymbol{z}), $$ the contribution is now $$
\text{Con}''(\si,\tau) = \begin{cases} \sgn(\si)\sgn(\tau), \ \
&\text{if each}\ \ \la_j-j-(\mu_{\ts}-\ts)\in \{0, 1\}\\ 0, \
&\text{else}, \end{cases} $$ and $$
Q_n(\kappa_{\be^-})(\lambda,\mu)=\frac{\dim\mu}{\dim\lambda}\frac{1}{n!}
\frac{p^{\sum_{i=1}^n \la_i-\mu_i}}{(1+p)^n}
\sum_{\sigma,\tau}\text{Con}''(\si,\tau) . $$ From here, the
proof is the essentially identical as the $\be^+$ case, except
with negative signs inserted and inequalities reversed.
\end{proof}

\begin{lemma}\label{SecondStepQuantum} If Theorem \ref{QRW}
holds for two functions $F_1$ and $F_2$, then it holds for
$F_1F_2$. \end{lemma} \begin{proof} This follows from
Proposition \ref{QProp}. \end{proof}

\begin{lemma}\label{ThirdStepQuantum} If Theorem \ref{QRW} holds
for a sequence of functions ${F_k}$ which converge uniformly to
a function $F$ on $A$, then the theorem also holds for $F$.
\end{lemma} \begin{proof} With $f_k$ defined as in
\eqref{DefnOff}, it is immediate that $f_k$ converges to $f$
uniformly. Since the determinant is a continuous function of its
entries, $T_n(\boldsymbol{1};F_k)(\la,\mu)$ converges to
$T_n(\boldsymbol{1};F)(\la,\mu)$. By \eqref{Qn},
$Q_n(\kappa_{F_k})(\la,\mu)$ converges to
$Q_n(\kappa_F)(\la,\mu)$ as well. \end{proof}

Finally, since $e^{x}=\lim_{k\rightarrow\infty} (1+ x/k)^k =
\lim_{k\rightarrow\infty} (1-x/k)^{-k}$, Lemmas
\ref{FirstStepQuantum}, \ref{SecondStepQuantum} and
\ref{ThirdStepQuantum} finish the proof of Theorem \ref{QRW}.
\end{proof}

Let us also prove a statement similar to Theorem \ref{Theorem}.

\begin{proposition} For $\kappa\in L^2(G,\C)^G$,
\begin{equation*} \sum_{\mu\in\GT_n}
Q_n(\kappa)(0,\mu)\LR\frac{\dim\tau}{\dim\la\dim\mu} =
Q_n(\kappa)(\lambda,\tau). \end{equation*} \end{proposition}
\begin{proof} By linearity, it suffices to prove the result when
$\kappa=\chi_{\be}$. By \eqref{Qn1}, \begin{align*}
\sum_{\mu\in\GT_n}
Q_n(\chi_{\be})(0,\mu)\LR\frac{\dim\tau}{\dim\la\dim\mu} &=
\sum_{\mu\in\GT_n}
\frac{\dim{\tau}}{\dim{\lambda}}\LR\int_{U(n)}
\overline{\chi_{\mu}(g)}\chi_{\be}(g)dg\\
&=\frac{\dim{\tau}}{\dim{\lambda}} c_{\la\be}^{\tau}.
\end{align*} On the other side, \begin{align*}
Q_n(\chi_{\be})(\lambda,\tau) &=
\frac{\dim{\tau}}{\dim{\lambda}}\int_{U(n)}
\chi_{\lambda}(g)\overline{\chi_{\tau}(g)}\chi_{\be}(g)dg\\
&=\frac{\dim{\tau}}{\dim{\lambda}}\int_{U(n)}
\overline{\chi_{\tau}(g)}\sum_{\mu\in\W} c_{\la\be}^{\mu} \cdot
\chi_{\mu}(g)dg \\ &=\frac{\dim{\tau}}{\dim{\lambda}}
c_{\la\be}^{\tau}. \end{align*} \end{proof}

\subsection{Restriction to Maximal Torus} The purpose of this
subsection is to demonstrate that there is a natural
representation theoretic reason for the occurrence of tensor
products in the transition probabilities. To see this, we will
consider the restriction of the quantum random walk to the von
Neumann algebra of the maximal torus. This is a natural
restriction to consider: in \cite{KOR}, it is shown that the
Krawtchouk and Charlier processes are Doob $h$--transforms of
Bernoulli and exponential random walks; while in \cite{kn:Bi2},
it is shown that for representations whose highest weight is miniscule, the restriction of the quantum random walk to the
center is the Doob $h$--transform of the restriction to the
maximal torus.

Restricting the highest weight representation $V_{\la}$ to
$\T^n$ yields a decomposition into one--dimensional subspaces
\begin{equation}\label{Decomp} V_{\la}=\bigoplus_{x\in P}
U_x^{\oplus n_{\la}(x)}, \end{equation} where $$ U_x=\{v\in
V_{\la}: \bth\cdot v = x(\bth)v \quad \text{for all} \quad
\bth\in \T^n\} $$ and $n_{\la}(x)$ are non--negative integers.
In terms of characters, this means that $$
\chi_{\la}(\bth)=\sum_{x\in P} n_{\la}(x)x(\bth). $$ For
$\kappa\in L^2(G,\C)^G$, define $n_{\kappa}(x)$ by linear
extension, i.e. $$ n_{\kappa}(x) = \sum_{\la\in\W}
\widehat{\kappa}(\la) n_{\la}(x). $$

Let $\mathcal{T}$ be the subalgebra of $vN(G)$ generated by
$\{\al(\theta):\theta\in\T^n\}$. Since every element of $G$ is
conjugate to exactly one element of $\T^n$, we can decompose the
Haar measure on $G$ as a measure on $\T^n\times\T^n\backslash
G$. Thus $L^2(G,dg)\cong L^2(\T^n,d\theta)\otimes L^2(\T^n
\backslash G)$, where $d\theta$ is Haar measure on $\T^n$. With
this isomorphism, $\alpha(\theta)$ acts as the identity element
on $L^2(\T^n \backslash G)$. Therefore $\mathcal{T}$ is
isomorphic to the group von Neumann algebra of $\T^n$.

Since the character group of $\T^n$ is isomorphic to $P$, there
is an isomorphism of $W^*$--algebras $\ze:\mathcal{T}\rightarrow
L^{\infty}(P)$ such that $\ze(\al(\theta))$ sends $x\in P$ to
$e(x)(\theta)$. Since $Q_n(\kappa)$ sends $\mathcal{T}$ to itself, the map
$\ze\circ Q_n(\kappa)\circ \ze^{-1}$ defines a classical Markov chain with
state space $P$, assuming that $\kappa$ is normalized so that $\kappa(\mathrm{Id})=1$. Identify $P$ with $\Z^n$ naturally, and write
$P_n(\kappa)(x,y),x,y\in \Z^n$ for the transition matrix of this
Markov chain.

\begin{proposition}\label{Exchange} 1. For any $\kappa\in
L^2(G,\C)^G$, \begin{equation}\label{Peqn} P_n(\kappa)(x,y) =
n_{\kappa}(y-x) \end{equation} Furthermore, for any $\sigma$ in
the Weyl group, $P_n(\kappa)(x,y)=P_n(\kappa)(\sigma x,\sigma
y)$.

2. The map $P_n:BC(G,\C)^G\rightarrow Mat(P\times P)$ is a
morphism of $^*-$ algebras.

\end{proposition} \begin{proof}

1. By Proposition 3.1 in \cite{kn:B}, \begin{equation}\label{Pn}
P_n(\kappa)(x,y)= \int_{\T^n}
e(x)(\theta)\overline{e(y)(\theta)}\kappa(\theta)d\theta,
\end{equation} which implies that \begin{align*} P_n\left(
\kappa\right)(x,y) &=\int_{\T^n} \overline{e(y-x)(\theta)}
\sum_{\la\in\W} \widehat{\kappa}(\la)\chi_{\la}(\theta)d\theta\\
&= \int_{\T^n}
\overline{e(y-x)(\theta)}\sum_{\la\in\W}\widehat{\kappa}(\la)\sum_{z\in
P} n_{\la}(z)\cdot e(z)(\theta)d\theta\\ &=\int_{\T^n}
\overline{e(y-x)(\theta)}\cdot
e(y-x)(\theta)\sum_{\la\in\W}\widehat{\kappa}(\la) n_{\la}(y-x)
d\theta\\ &= \sum_{\la\in\W} \widehat{\kappa}(\la) n_{\la}(y-x)
= n_{\kappa}(y-x). \end{align*} Furthermore, since the weight
multiplicities are invariant under the action of the Weyl group,
it follows that the transition probabilities are invariant under
the Weyl group.

2. The fact that $P_n$ is linear and preserves $^*$ follows from
\eqref{Pn}. Since $\sum_{z\in P}
e(z)(\theta)\overline{e(z)(\theta')}$ is the Dirac delta
$\delta_{\theta\theta'^{-1}}$, it follows that from \eqref{Pn}
multiplication is also preserved. This can also be seen from the
construction of the quantum random walk, as in the proof of
Proposition \ref{QProp}.2.

There is also a proof which illuminates the occurrence of tensor
products. To show that $P_n$ preserves multiplication, by
\eqref{Peqn} it suffices to show that the map $n:
BC(G,\C)^G\rightarrow B(P,\C)$ defined by $n(\kappa)=n_{\kappa}$
from bounded, continuous complex--valued class functions on $G$
to bounded complex--valued functions on $P$ is a morphism of
$^*$--algebras, where the multiplication in $B(P,\C)$ is usual
convolution. By definition, $n$ is linear, so it suffices to
show that $$ n_{\chi_{\la}\chi_{\mu}}=n_{\chi_{\la}}*
n_{\chi_{\mu}}. $$ Letting $W(\pi)$ denote the multiset of
weight multiplicities (i.e. the number of times that $x\in P$
appears in $W(\pi)$ is $n_{\chi_{\pi}}(x)$, which is the
multiplicity of the weight $x$ in the representation $V_{\pi}$),
this is equivalent to $$ W(\pi_1\otimes\pi_2) = W(\pi_1)+
W(\pi_2), $$ where $A+B$ denotes the usual addition of
multisets, $A+B=\{a+b:a\in A,b\in B\}$. However, by
\eqref{Decomp}, this follows immediately. \end{proof}

\bibliographystyle{plain}  
\begin{thebibliography}{99)}
\bibitem{kn:B} Biane P. \textit{Quantum random walk on the dual
of} $SU(n).$ Probab. Th. Rel. Fields 89, 117--129 (1991).
\bibitem{kn:Bi2} Biane P. \textit{Miniscule weights and random
walks on lattices.} Quantum Probability and Related Topics Vol.
VII, 51--65 (1992). \bibitem{kn:BF} Borodin, A.; Ferrari, P.L.
\textit{Anisotropic growth of random surfaces in 2+1
dimensions}. J. Stat. Mech. (2009) P02009.
\href{http://arxiv.org/abs/0804.3035}{arXiv:0804.3035v1}
\bibitem{BF2} Borodin, A.; Ferrari, P.L. \textit{Large time
asymptotics of growth models on space-like paths I: PushASEP}.
Elec. J. Prob, Bolume 13, Number 50 (2008), 1380--1418.
\href{http://arxiv.org/abs/0707.2813}{arXiv:0707.2813}
\bibitem{BC} Borodin, A; Corwin, I. \textit{Macdonald processes}
\href{http://arxiv.org/abs/1111.4408}{arXiv:1111.4408}
\bibitem{BK0} Borodin, A.; Kuan, J. \textit{Asymptotics of
Plancherel measures for the infinite-dimensional unitary group}.
Adv. Math. \textbf{219} (2008), 894--931.
\href{http://arxiv.org/abs/0712.1848v1}{arXiv:0712.1848v1}
\bibitem{C} Corwn, I. \textit{The Kardar-Parisi-Zhang equation
and universality class}.
\href{http://arxiv.org/abs/1106.1596}{arXiv:1106.1596}
\bibitem{kn:D} Defosseux, M. \textit{An interacting particle
model and a Pieri-type formula for the orthogonal group}. J.
Theor. Probab, Feb 2012.
\href{http://arxiv.org/abs/1012.0117}{arXiv:1012.0117v1}
\bibitem{kn:D2} Defosseux, M. \textit{Interacting particle
models and the Pieri-type formulas : the symplectic case with
non equal weights}.
\href{http://arxiv.org/abs/1104.4457}{http://arxiv.org/abs/1104.4457}
\bibitem{kn:Di} Dixmier, J. \textit{Les C *-algèbres et leurs
représentations}. Paris: Gauthier-Villars 1964 \bibitem{kn:FH}
Fulton, W.; Harris, J. \textit{Representation theory: a first
course}. Graduate Texts in Mathematics, Vol. 129. Springer, New
York, 1991. 
\bibitem{KOR} K\"{o}nig, W; O'Connell, N; Roch, S;
\textit{Non--colliding random walks, tandem queues and discrete
orthogonal polynomial ensembles}, Elec. J. Prob, Volume 7,
Number 1 (2002), 1--24. 
\bibitem{kn:K} Kuan, J.
\textit{Discrete-time particle system with a wall and
representations of O(infinity) }.
\href{http://arxiv.org/abs/1203.1660}{arXiv:1203.1660v1}
\bibitem{K2} Kuan, J. 
\textit{A (2+1)--dimensional Gaussian field as fluctuations of quantum random walks on quantum groups.}
\href{http://arxiv.org/abs/1601.04402}{arXiv:1601.04402v1}
\bibitem{WW} Warren, J; Windridge, P. \textit{Some Examples of
Dynamics for Gelfand--Tsetlin Patterns.} Elec. J. Prob, Volume
14, Number 59 (2009), 1745--1769.
\href{http://arxiv.org/abs/0812.0022v2}{arXiv:0812.0022}
\end{thebibliography}
\end{document}